\documentclass[11pt]{amsart}
\usepackage{amsmath,amsthm,amssymb,latexsym,color}
\usepackage{hyperref}
\usepackage{lipsum}
\usepackage{todonotes}
\usepackage{multirow}
\usepackage[mathscr]{eucal}

\usepackage{xcolor}
\date{}
\usepackage{pgfplots}
\usepackage{tikz}
\usetikzlibrary{arrows,matrix,positioning}
\usepackage{graphicx}
\usepackage{subfigure}

\title{On bounded degree graphs with large size-Ramsey numbers}
\author{
Konstantin Tikhomirov
}
\address{
\medskip
\noindent
Department of Mathematical Sciences\\
Carnegie Mellon University\\
Wean Hall 6113\\
Pittsburgh, PA 15213\\
\texttt{\small
e-mail:   ktikhomi@andrew.cmu.edu}
}

\thanks{The work is partially supported by the NSF Grant DMS 2054666}

\newtheorem{theorem}{Theorem}[section]
\newtheorem*{theorem*}{Theorem}

\newtheorem{lemma}[theorem]{Lemma}
\newtheorem{cor}[theorem]{Corollary}
\newtheorem{defi}[theorem]{Definition}

\theoremstyle{definition}
\newtheorem{Remark}[theorem]{Remark}
\newtheorem*{Remark*}{Remark}

\theoremstyle{plain}

\newcommand{\Event}{\mathcal{E}}
\def\Prob{{\mathbb P}}

\begin{document}

\maketitle

\begin{abstract}
The size-Ramsey number $\hat r(G')$ of a graph $G'$ is defined as the smallest integer $m$ so that
there exists a graph $G$ with $m$ edges such that every $2$--coloring of the edges of $G$
contains a monochromatic copy of $G'$.
Answering a question of Beck, R\"{o}dl\ and\ Szemer\'{e}di showed that for every $n\geq 1$
there exists a graph $G'$ on $n$ vertices each of degree at most
three, with size-Ramsey number at least $cn\log^{\frac{1}{60}}n$
for a universal constant $c>0$. 
In this note we show that a modification of R\"{o}dl\ and\ Szemer\'{e}di's construction leads to a bound
$\hat r(G')\geq cn\,\exp(c\sqrt{\log n})$.
\end{abstract}

\section{Introduction}

The size-Ramsey number of a graph $G'$ can be viewed as the number of edges
in a most economical ``robust version'' of $G'$, a graph $G$ such that every $2$--coloring of the edges of $G$
contains a monochromatic copy of $G'$ \cite{EFRS1978}.
In \cite{Beck1990}, Beck asked whether every bounded degree graph has size-Ramsey number linear
in the number of its vertices. The question was answered negatively by R\"{o}dl\ and\ Szemer\'{e}di
in \cite{RS2000} who constructed for every $n\geq 1$ a graph $G'$ on $n$ vertices
with the maximum degree at most three, such that $\hat r(G')\geq cn\log^{\frac{1}{60}}n$.
The authors of \cite{RS2000} further conjectured that for every $d\geq 3$ there is
a number $\varepsilon=\varepsilon(d)>0$ such that for all sufficiently large $n$,
$$
n^{1+\varepsilon}\leq
\max\big\{\hat r(G'):\;G'\mbox{ has $n$ vertices, each of degree at most $d$}\big\}
\leq n^{2-\varepsilon}.
$$
The conjecture was partially confirmed in \cite{KRSS2011}
where it was shown that
$$
\max\big\{\hat r(G'):\;G'\mbox{ has $n$ vertices, each of degree at most $d$}\big\}
\leq n^{2-1/d+o(1)}.
$$
We refer to papers \cite{CNT2022,DP2022} for improvements of the upper bound
as well as further references regarding size-Ramsey numbers.
Whereas there has been substantial progress on estimating size-Ramsey numbers of
sparse graphs from above, the lower bound $cn\log^{\frac{1}{60}}n$ given
in \cite{RS2000} seems to be the best known estimate as of this writing.
The purpose of this note is to show that a modification of R\"{o}dl\ and\ Szemer\'{e}di's construction
leads to an improvement of the lower bound:
\begin{theorem}\label{alijnfaifjnfiwjn}
For every $n\geq 1$ there is a graph $G'$ on $n$ vertices of maximum degree at most three such that
$\hat r(G')\geq cn\,\exp(c\sqrt{\log n})$, for a universal constant $c>0$.
\end{theorem}

\section{Proof of Theorem~\ref{alijnfaifjnfiwjn}}

\begin{defi}
Let $k\geq 2$ be an integer parameter. We
define a labeled random graph $U_k=(V_k,E_k)$ as follows.
Let $T = (V_k,E_T)$
be a complete rooted binary tree of depth $k$,
and let $V_L \subset V_k$
be its set of leaves.
Let $C = (V_L,E_C)$
be a spanning cycle on $V_L$ chosen uniformly at random.
Then we let $U_k$ be the union of $T$ and $C$,
namely the graph with vertex set $V_k$ and edge set $E_T \cup E_L$.
We refer to Figure~\ref{aljkfnalkfjnslkfjnlk} for a realization of $U_3$.
\end{defi}

\begin{Remark}
We will call the root of $T$
the {\bf root} of $U_k$. 
\end{Remark}

\begin{figure}[h]
\caption{A realization of $U_3$.}

\centering  

\label{aljkfnalkfjnslkfjnlk}

\subfigure
{
\begin{tikzpicture}[every node/.style={circle,thick,draw}]
    \node (1) at (0,-3) {root};
    \node (2) at (-2,-2) {};
    \node (3) at (2,-2) {};
    \node (4) at (-3,-1) {};
    \node (5) at (-1,-1) {};
    \node (6) at (1,-1) {};
    \node (7) at (3,-1) {};
    \node (8) at (-3.5,1) {};
    \node (9) at (-2.5,1) {};
    \node (10) at (-1.5,1) {};
    \node (11) at (-0.5,1) {};
    \node (12) at (0.5,1) {};
    \node (13) at (1.5,1) {};
    \node (14) at (2.5,1) {};
    \node (15) at (3.5,1) {};
    \draw[blue, very thick] [-] (1) -- (2);
    \draw[blue, very thick] [-] (1) -- (3);
    \draw[blue, very thick] [-] (2) -- (4);
    \draw[blue, very thick] [-] (2) -- (5);
    \draw[blue, very thick] [-] (3) -- (6);
    \draw[blue, very thick] [-] (3) -- (7);
    \draw[blue, very thick] [-] (4) -- (8);
    \draw[blue, very thick] [-] (4) -- (9);
    \draw[blue, very thick] [-] (5) -- (10);
    \draw[blue, very thick] [-] (5) -- (11);
    \draw[blue, very thick] [-] (6) -- (12);
    \draw[blue, very thick] [-] (6) -- (13);
    \draw[blue, very thick] [-] (7) -- (14);
    \draw[blue, very thick] [-] (7) -- (15);
    \draw[blue, very thick] (8) to [bend left=70] (10);
    \draw[blue, very thick] (10) to [bend left=70] (12);
    \draw[blue, very thick] (12) to [bend left=60] (14);
    \draw[blue, very thick] (11) to [bend left=70] (14);
    \draw[blue, very thick] (9) to [bend left=60] (11);
    \draw[blue, very thick] (9) to [bend left=70] (15); 
    \draw[blue, very thick] (13) to [bend left=70] (15);
    \draw[blue, very thick] (8) to [bend left=70] (13);    
    \end{tikzpicture}
}

\end{figure}
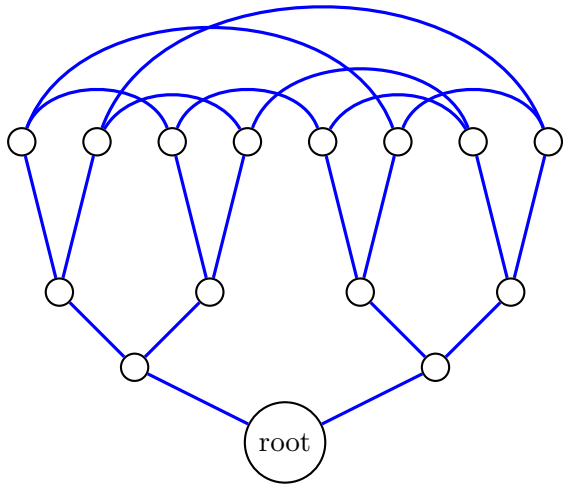

\begin{lemma}
Let $U_k=(V_k,E_k)$ be as above,
and let $G=(V_G,E_G)$ be a non-random labeled graph of maximum degree $d$.
Fix any vertex $v\in V_G$. Then
$$
\Prob\big\{\mbox{There is an embedding of $U_k$ into $G$ mapping the root of $U_k$ into $v$}\big\}\leq 
\frac{d^{2^k-1}\cdot d^{2^{k+1}}}{(2^k-1)!}.
$$
\end{lemma}
\begin{proof}
Let $T=(V_k,E_T)$ be the binary tree subgraph of $U_k$ from the definition, let $v_r$
be its root, and $V_L$ be its set of leaves.
We first claim that for any non-random embedding $\phi$ of $T$ into $G$ mapping $v_r$ into $v$,
$\phi$ can be extended to an embedding of $U_k$ into $G$ with probability at most
$$
\prod_{\ell=1}^{2^k-1}\frac{d}{\ell}=\frac{d^{2^k-1}}{(2^k-1)!}.
$$
This bound can be obtained by considering the following spanning cycle generation on $V_L$.
Let $v_0\in V_L$ be a fixed vertex of $V_L$. At the first step, a vertex $v_1\in V_L\setminus\{v_0\}$
is chosen uniformly at random; at the second step, $v_2$ is chosen uniformly on the set $V_L\setminus\{v_0,v_1\}$,
and so on. The cycle is given by the random set of edges $\{v_i,v_{(i+1)\mod |V_L|}\}$, $0\leq i\leq |V_L|-1$.
Then, conditioned on any realization of $v_1,\dots,v_i$,
the probability that $\phi(v_i)$ and $\phi(v_{i+1})$ are adjacent in $G$ is at most $\frac{d}{2^{k}-1-i}$,
and the claim follows.

To complete the proof of the lemma, it is sufficient to give upper bound on the number $N$ of 
embeddings $\phi$ of $T$ into $G$ mapping $v_r$ into $v$.
Since every vertex of $G$ has at most $d$ neighbors, a rough bound gives
$$
N\leq d^{2^1+2^2+\dots+2^k}\leq d^{2^{k+1}},
$$
and the result follows.
\end{proof}

As an immediate corollary, we get
\begin{cor}\label{dfajfnakfjkjnlfkjn}
Let $r\geq 1$, $k,d\geq 2$, and let
$U^{(1)},\dots,U^{(r)}$ be i.i.d copies of $U_k$. Then
\begin{align*}
\Prob\big\{&\mbox{There is
a labeled graph $G$ of maximum degree $d$
and a vertex $v$ of $G$ such that}\\
&\mbox{for each $i\leq r$ there is an embedding of $U^{(i)}$ into $G$ mapping the root of $U^{(i)}$ into $v$}\big\}\\
&\leq 
\bigg(\frac{d^{2^k-1}\cdot d^{2^{k+1}}}{(2^k-1)!}\bigg)^r\cdot \big(d^{k+1}\big)^{d\cdot d^{k+1}}.
\end{align*}
\end{cor}
\begin{proof}
The proof is accomplished via a union bound.
We first note that the event in question coincides with the event
\begin{align*}
\big\{&\mbox{There is
a labeled graph $G$ of maximum degree $d$ {\it{}on $d^{k+1}$ vertices}
and a vertex $v$ of $G$ such}\\
&\mbox{that for each $i\leq r$ there is an embedding of $U^{(i)}$ into $G$ mapping the root of $U^{(i)}$ into $v$}\big\}.
\end{align*}
Indeed, the claim follows by observing that in any graph of maximum degree $d$, any ball of radius $k$
contains at most $1+d+d^2+\dots+d^{k}\leq d^{k+1}$ vertices.
We further can assume that the graphs $G$ in the above event
have a common vertex set $V$. 
The total number of such graphs $G$ can be crudely bounded above by
$$
\big(d^{k+1}\big)^{d\cdot d^{k+1}},
$$
implying the result.
\end{proof}

\bigskip

Everything is ready for the proof of Theorem~\ref{alijnfaifjnfiwjn}.
Note that the result is trivial for small $n$ by choosing the constant $c$ in the statement of the theorem sufficiently small.
From now on, we assume that $n$ is a large integer. Let 
parameters $h\geq 1$, $1\leq r\leq h$, $k,d\geq 2$ be chosen as follows:
$$
d:=\big\lfloor\exp\big(\sqrt{\log n}/100\big)\big\rfloor;\quad
k=\big\lfloor\sqrt{\log n}/10\big\rfloor;\quad
r:= d\cdot d^{k+1};\quad
h:=\lfloor 2^{-k-1} n\rfloor.
$$
Let $U^{(1)},\dots,U^{(h)}$ be i.i.d copies of $U_k$ (on disjoint vertex sets),
and define a random graph $G'$ as the union of $U^{(1)},\dots,U^{(h)}$.
We will show that with a positive probability, $\hat r(G')\geq \exp\big(\sqrt{\log n}/1000\big)n$.
For every $r$--subset $S$ of $\{1,2,\dots,h\}$, let $\Event_S$ be the event
\begin{align*}
\Event_S:=
\big\{&\mbox{There is
a labeled graph $G$ of maximum degree $d$
and a vertex $v$ of $G$ such that}\\
&\mbox{for each $i\in S$ there is an embedding of $U^{(i)}$ into $G$ mapping the root of $U^{(i)}$ into $v$}\big\},
\end{align*}
and let $\Event$ be the intersection of the complements $\Event_S^c$, $|S|=r$, $S\subset \{1,2,\dots,h\}$.
In view of Corollary~\ref{dfajfnakfjkjnlfkjn} and our choice of parameters,
\begin{align*}
\Prob(\Event)
&\geq
1-\sum\limits_{S\subset \{1,\dots,h\},\,|S|=r}\Prob(\Event_S)\\
&\geq 1-{h\choose r}\bigg(\frac{d^{2^k-1}\cdot d^{2^{k+1}}}{(2^k-1)!}\bigg)^r\cdot \big(d^{k+1}\big)^{d\cdot d^{k+1}}\\
&\geq
1- \bigg(\frac{eh\cdot d^{2^k-1}\cdot d^{2^{k+1}}}{d(2^k-1)!}\bigg)^r\\
&\geq
1- \bigg(\frac{e\cdot n\cdot 
\exp\big(2^k-1+2^{k+2}\sqrt{\log n}/100\big)}{(2^k-1)^{2^k-1}}\bigg)^r
>0.
\end{align*}
Condition on any realization of $G'$ from $\Event$.
Let $G=(V,E)$ be any graph with at most $\exp\big(\sqrt{\log n}/1000\big)\,n$ edges.
We will show that there is a $2$--coloring of the edges of $G$ such that $G$ does not contain a monochromatic
copy of $G'$.

Denote by $E_{high}\subset E$ the collection of all edges of $G$ incident to vertices of degree
at least $d+1$, and let $\tilde G$ be the subgraph of $G$ obtained by
removing $E_{high}$ from the edge set of $G$. Observe that the maximum degree of $\tilde G$
is at most $d=\lfloor\exp\big(\sqrt{\log n}/100\big)\rfloor$.
By the definition of $\Event$, for every vertex $v$ of 
$\tilde G$ there are at most $r-1$ indices $i\leq h$ such that $U^{(i)}$ can be embedded into $\tilde G$
with the root of $U^{(i)}$ mapped into $v$.
Since the number of non-isolated
vertices of $\tilde G$ is at most $2\exp\big(\sqrt{\log n}/1000\big)\,n$
we get that there exists an index $i_0\leq h$ and a subset of vertices
$V_r$ of $\tilde G$ of size at most 
$r\cdot 2\exp\big(\sqrt{\log n}/1000\big)\,n/h$,
such that $U^{(i_0)}$ can be embedded into $\tilde G$
only when mapping the root of $U^{(i_0)}$ into one of vertices in $V_r$.

At this stage, we can define a coloring of $G$.
Color all edges from $E_{high}$ as well as all edges incident to $V_r$ red, and all other edges blue,
and denote the corresponding sets of edges by $E_{red}$ and $E_{blue}$, respectively.
Note that the blue subgraph $G_{blue}=(V,E_{blue})$ of $G$ is also a subgraph of $\tilde G$, and $G_{blue}$
cannot contain a copy of $G'$ since, by our construction,
it does not contain a copy of $U^{(i_0)}$.
Assume for a moment that $G'$ is embeddable into subgraph $G_{red}=(V,E_{red})$,
and let $\phi:G'\to G_{red}$ be an embedding.
Every edge from $E_{red}\setminus E_{high}$ is incident to a vertex in $V_r$ which has degree at most $d$ in $G$,
and therefore
\begin{align*}
|E_{red}\setminus E_{high}|&\leq d\cdot r\cdot 2\exp\big(\sqrt{\log n}/1000\big)\,n/h\\
&\leq 4\exp\big((k+3)\sqrt{\log n}/100+\sqrt{\log n}/1000\big)\,2^{k+1}
< h/2.
\end{align*}
Let $I\subset\{1,2,\dots,h\}$ be the subset of all indices $i$ such that $\phi(U^{(i)})$
contains an edge from $E_{red}\setminus E_{high}$. Then, by the above,
$$
|I|< h/2.
$$
For every $i\in \{1,2,\dots,h\}\setminus I$, the edge set of the graph $\phi(U^{(i)})$ 
is entirely comprised by $E_{high}$ and, in particular,
more than $2^{k-1}$ vertices of $\phi(U^{(i)})$ have degree at least $d+1$ in $G$.
Thus, the total number of vertices in $G$ of degree $d+1$ or larger can be estimated from below by
$$
(h-|I|)\cdot 2^{k-1}
\geq
\frac{h}{2}\cdot 2^{k-1}.
$$
On the other hand, the number of such vertices cannot be greater than
$$
\frac{2\exp\big(\sqrt{\log n}/1000\big)\,n}{d}.
$$
We get the inequality
$$
\frac{h}{2}\cdot 2^{k-1}\leq \frac{2\exp\big(\sqrt{\log n}/1000\big)\,n}{d},
$$
which is clearly false. The contradiction shows that $G'$ cannot be embedded into $G_{red}$, and the result follows.

\end{document}